\documentclass[a4paper,11pt]{article}
\usepackage[plainpages=false]{hyperref}
\usepackage{amsfonts,latexsym,rawfonts,amsmath,amssymb,amsthm,mathrsfs}
\usepackage{amsmath,amssymb,amsfonts,latexsym,lscape,rawfonts}

\usepackage[all]{xy}
\usepackage{eufrak}
\usepackage{makeidx}         
\usepackage{graphicx,psfrag}

\usepackage{array,tabularx}

\usepackage{setspace}

\newtheorem{thm}{Theorem}[section]

\newtheorem{lem}[thm]{Lemma}
\newtheorem{clm}[thm]{Claim}

\theoremstyle{remark}
\newtheorem{rmk}[thm]{Remark}

\theoremstyle{definition}
\newtheorem{Def}[thm]{Definition}                                        %

\title{On the K\"ahler-Ricci flows near the Mukai-Umemura 3-fold.}
\author{Yuanqi Wang}

\begin{document}
\maketitle
\begin{abstract}
In this short note, we show that given a special K\"ahler-Einstein degeneration with bounded geometry, for any noncentral fiber, there exists a K\"ahler-Ricci flow which converges to the K\"ahler-Einstein metric of the central fiber. As an example, Tian's deformations \cite{Tian97} of the Mukai 3-fold admit K\"ahler-Ricci flows which converge to Donaldson's K\"ahler-Einstein metric \cite{Don08} over the  Mukai 3-fold.
\end{abstract}
\section{Introduction}This is a following up note of \cite{SW}. 
In 1982, R.Hamilton introduced the Ricci flow
in \cite{Hamilton}. 
 In the K\"ahler setting,  the
K\"ahler condition is preserved under the flow. H-D. Cao(\cite{Cao}) first studied the K\"ahler-Ricci flow.
Using Yau's estimates for complex Monge-Amp\`ere equations,  Cao
proved when $C_1(M)$ is negative or zero, the properly normalized K\"ahler-Ricci flow converges to a K\"ahler-Einstein metric, thus obtained another proof of the Calabi conjecture originally proved by Yau in \cite{Yau}. 

When $C_1(M)>0$, the situation is more complicated. A result announced by Perelman, and   proved by Tian-Zhu in \cite{TZ1} says that if there exists a K\"ahler-Einstein metric in $C_{1}(M)$ with respect to the underline complex structure, then the K\"ahler-Ricci flow
\begin{equation}\label{KRF flow equation }
\frac{\partial g}{\partial t}=-Ric+g
\end{equation}
initiated from any metric in $C_{1}(M)$ converges to the  K\"ahler-Einstein metric up to a biholomorphism. However, when there is no 
K\"ahler-Einstein metric with respect to the underlying complex structure, it's very interesting to study where does the flow go. It's widely believed that even the  limit complex structure (as $t\rightarrow\infty$ ) could be different from the complex structure of the initial metric of the flow.

 In this note, we show that there really exist such kind of  examples of flows with jumping complex structures as $t\rightarrow \infty$. These examples live on certain deformations of  the well known Mukai-Umemura 3-fold (see \cite{Mukai}, \cite{MukaiUmemura}). In the appendix we will say more about these 3-folds for the sake of a self-contained proof, but in this section  we  abbreviate everything to make it easiest to understand. Denote $M_{F_0}$ as the Mukai-Umemura 3-fold, and $M_{F_{v}}$ as the deformations (parametrized by tensors $v$) which are  considered by Tian in \cite{Tian97}. The amazing history is that: on one hand, Tian proved that $M_{F_{v}}$ does not admit any K\"ahler-Einstein metric; on the other hand, Donaldson proved in \cite{Don08} that $M_{F_0}$ admits a K\"ahler-Einstein metric. Moreover, Donaldson also  showed there is a 2-parameter family $M_{\alpha,C}$ of deformations (of $M_{F_0}$) which do admit K\"ahler-Einstein metrics. Though $M_{F_0}$ is not isolated K\"ahler-Einstein,   the following theorem holds. 
\begin{thm}\label{Main Theorem}For any $v$ close enough to $0$ (as a tensor), there exists a metric $\omega_v$ over Tian's deformations $M_{F_{v}}$ , such that the K\"ahler-Ricci flow initiated from
$\omega_v$ converges to Donaldson's K\"ahler-Einstein metric over the Mukai 3-fold $M_{F_0}$ (up to a biholomorphism). In particular, the limit complex structure (as $t\rightarrow \infty$) is different from the complex structure of the initial metric $\omega_v$.
\end{thm}

Actually it's well known that $M_{F_{v}}$ is a special example of the following  degenerations. 
\begin{Def}\label{Definition of the degeneration}We call the triple $(\mathfrak{M},\varpi,B)$ a special K\"ahler-Einstein degeneration with bounded geometry if it satisfies the following. 
\begin{enumerate}
\item $(\mathfrak{M},\varpi,B)$ is a differentiable family of complex manifolds in the sense of Definition 4.1 in  \cite{Kodaira} ($\varpi$ is a map of full rank from $\mathfrak{M}$ to $B$), $B$ is a ball in  $R^{m}$ and is centered at the origin. 
\item For any $u\in B$, $\varpi^{-1}(u)\triangleq M_{u}$ is a Fano K\"ahler manifold. There exists a smooth tensor $J\in \Gamma (T^{\star}\mathfrak{M}\otimes T\mathfrak{M})$, such that for any point $p\in \mathfrak{M}$, $T_{p}M_{\varpi(p)}$ is an invariant subspace of $J$, and $J$ restricts to  the complex structure of $M_{\varpi(p)}$. Moreover, there exists a smooth two-form $G$ over $\mathfrak{M}$ such that restricted to each fiber $(M_{u}, TM_{u})$, $G|_{(M_{u}, TM_{u})}=\omega_u$ is a K\"ahler metric form in $ C_{1}(M_u,J_u)$.
\item $M_{0}$ admits a K\"ahler-Einstein metric $\omega_{KE}$.
 \item There exists a  smooth vector field $\sigma$ over $\mathfrak{M}$, such that $\varpi_{\star}\sigma=-\Sigma_{i}u_{i}\frac{\partial}{\partial u_i}$. Furthermore, the diffeomorphism $\sigma_{a}$ generated by $\sigma$ is a biholomorphism from $M_{u}$ to $M_{e^{-a}u}$.
\end{enumerate}
\end{Def}
 With respect to Definition (\ref{Definition of the degeneration}), the  following more general theorem is true. 
\begin{thm}\label{thm convergence over degenerations}
Let $(\mathfrak{M},\varpi,B)$ be a  special K\"ahler-Einstein degeneration with bounded geometry. For any fiber $M_{u},\ u\in B,\ u\neq 0$, there exists a metric $\omega_{\phi,u}$, such that the K\"ahler-Ricci flow initiated from  $\omega_{\phi,u}$ converges in Cheeger-Gromov sense with polynomial rate to the K\"ahler-Einstein metric $\omega_{KE}$ of $M_0$, up to a biholomorphism. 
\end{thm}

\begin{rmk} Actually Theorem \ref{thm convergence over degenerations} implies more than Theorem \ref{Main Theorem} on the behavior of K\"ahler-Ricci flows near $M_{F_0}$, starting from the 4th kind of orbits as in page 44 of \cite{Don08}. Since these are exactly parallel to Theorem \ref{Main Theorem},  for the sake of brevity we only consider Tian's deformations.  Though holomorphic degenerations are considered most often, Theorem \ref{thm convergence over degenerations} holds more generally  for smooth degenerations. This is consistent with the real setting of Theorem 1.1 in \cite{SW}. The bounded geometry assumption is satisfied  in  most of the existing examples of degenerations. Similar topic  is also discussed in Tian-Zhu's work \cite{TZ2}. 
\end{rmk}

\textbf{Acknowledgements}: The author is grateful to Xiuxiong Chen for continuous 
support and encouragement. The author would like to thank   Song Sun for inspiring conversations. The author also would like to thank Xianzhe Dai and Bing Wang  for related discussions.
\section{Proof of Theorem \ref{thm convergence over degenerations}.}
The proof of Theorem \ref{Main Theorem} and Theorem \ref{thm convergence over degenerations} are  by  simple observations based on the work of Sun and the author  \cite{SW}, Donaldson \cite{Don08}, Tian \cite{Tian97}, and Chen-Sun  \cite{CS}. The point is to construct metrics which satisfy the requirements of Theorem 1.1 in \cite{SW}. First we have the following lemma which is  well known and essentially due to Kodaira \cite{Kodaira}. 
\begin{lem}\label{From the complex setting to the real setting of J} Suppose $(\mathfrak{M},\varpi,B)$ is a special K\"ahler-Einstein degeneration with bounded geometry.   Then for any $u=(u_1,....,u_{m})\in B$,  there is a family of smooth diffeomorphisms $\psi_{u}$ from 
$M_0$ to $M_{u}$ with the following property. For any $k>0$ we have \begin{equation}\label{Js close to J}
  \lim_{u\rightarrow 0}|\psi_{u,\star}^{-1}\circ J_{u}\circ \psi_{u,\star}-J_0|_{C^{k},g_0,M_0}=0,
  \end{equation}
 \begin{equation}\label{gs close to g0}
\lim_{u\rightarrow 0}|\psi_{u}^{\star}g_{u}-g_0|_{C^{k},g_0,M_0}=0.
\end{equation}

 Moreover, there is a smooth family of smooth diffeomorphisms $\psi_{s, u}$ from $\mathfrak{M}$ to itself such that 
 \begin{equation*}
 \varpi_{\star} (\frac{d \psi_{s, u}  }{ds})=\Sigma^{m}_{\beta=1}u_{\beta}\frac{\partial }{\partial t_{\beta}}   ;\  \psi_{1,u}|_{M_0}=\psi_{u}. 
 \end{equation*}
\end{lem}
\begin{proof}{of Lemma \ref{From the complex setting to the real setting of J}:} 
We work in the real setting. As the terminology of \cite{Kodaira}, given any point $p\in \mathfrak{M}$, we can choose   local coordinates 
\[(x_{b}, t_{\alpha}),\ 1\leq b\leq 2n,\ 1\leq \alpha\leq m.\]

Suppose $u=(u_1,....,u_{m})\in B$. The we can lift the vector field 
\[\Sigma_{\beta=1}^{m}u_{\beta}\frac{\partial }{\partial t_{\beta}}\] up to a vector field $V_{u}$ over $\mathfrak{M}$, as in the proof of Theorem 2.4 in \cite{Kodaira}. Since $\varpi_{\star}V_{u}=\Sigma_{\beta=1}^{m}u_{\beta}\frac{\partial }{\partial t_{\beta}}$ over $B$, then $V_{u}$ generates a smooth family of diffeomorphisms $\psi_{s,u}$ from $M_{0}$ to $M_{su}$. Namely, 
\begin{equation}\label{the lifted verctor field}
V_{u}=\Sigma_{\beta=1}^{m}u_{\beta}\Sigma_{k=1}^{k_0}\rho_{k}(\frac{\partial}{\partial t_{\beta}})_{k},
\end{equation}
where $\rho_{k}$ is the partition of unity subject to a open cover $U_{k}, 1\leq k\leq k_0$.
Moreover $\psi_{u}\triangleq \psi_{1,u}$ is a diffeomorphism from $M_{0}$ to $M_{u}$.

  Next we express the complex structure $J$ of $\mathfrak{M}$ in a coordinate neighborhood $U$ as the following.
\begin{eqnarray*} 
J&=&\Sigma_{b,c=1}^{2n}J_{bc}dx_{b}\otimes \frac{\partial }{\partial x_c}+ \Sigma_{b=1}^{2n}\Sigma_{\beta=1}^{m}J_{bt_{\beta}}dx_{b}\otimes \frac{\partial }{\partial t_{\beta}}
\\&+&\Sigma_{\alpha,\beta=1}^{m}J_{t_{\alpha}t_{\beta}}dt_{\alpha}\otimes \frac{\partial }{\partial t_{\beta}}+ \Sigma_{b=1}^{2n}\Sigma_{\beta=1}^{m}J_{t_{\beta}b}dt_{\beta}\otimes \frac{\partial }{\partial x_{b}}.
\end{eqnarray*}
Since $T_{p}M_{\varpi(p)}$ is an invariant subspace of $J$ for any $p$, we have 
\[J_{bt_{\beta}}=0,\ \textrm{for any}\ \beta,\ b.\]
Then restricted over $M_0$ and $TM_{0}$, using $g_0$ as our reference metric, we consider
  \[D=\psi_{u,\star}^{-1}\circ J_{u}\circ \psi_{u,\star}-J_0.  \]
   Let the coordinates representation of $D$ be 
\begin{eqnarray*} 
D&=&\Sigma_{b,c=1}^{2n}D_{bc}dx_{b}\otimes \frac{\partial }{\partial x_c}.
\end{eqnarray*}  Let $w_{b}=x_{b}\circ\psi_{u}$. Then for any $p\in M_0$ and $1\leq b,c\leq 2n$, we have
  \begin{equation}
  D_{bc}(p)=\Sigma_{\alpha,\beta=1}^{2n}\frac{\partial w_{\alpha}  }{\partial x_b}(p)\frac{\partial x_c  }{\partial w_{\beta}} (\psi_{u}(p))J_{\alpha \beta}(\psi_{u}(p))-J_{bc}(p).
  \end{equation}
  Using (\ref{the lifted verctor field}), we deduce
  \begin{equation*}
  \lim_{u\rightarrow 0}|V_u|_{C^{k},g_0,M_0}=0.
  \end{equation*}
  Hence 
    \begin{equation}\label{limit of psiu is the identity}
  \lim_{u\rightarrow 0}|\psi_{u}-id|_{C^{k},\mathfrak{M}}=0,\ \textrm{id is the identity diffeomorphism from}\  \mathfrak{M}\ \textrm{to}\  \mathfrak{M}.
  \end{equation}
  (\ref{limit of psiu is the identity})  implies
  \begin{eqnarray}\label{Js close to J in local coordinates}& &
  \lim_{u\rightarrow 0}|D_{bc}|_{C^{k},U}\nonumber
\\&=&\lim_{u\rightarrow 0}|\Sigma_{\alpha,\beta=1}^{2n}\frac{\partial w_{\alpha}  }{\partial x_b}(p)\frac{\partial x_c  }{\partial w_{\beta}} (\psi_{u}(p))J_{\alpha \beta}(\psi_{u}(p))-J_{bc}(p)|_{C^{k},U}
\nonumber
\\&=&0.
  \end{eqnarray}
  (\ref{Js close to J in local coordinates}) implies (\ref{Js close to J}). Similarly, (\ref{gs close to g0}) is true. 

\end{proof}

\begin{lem}\label{perturbation of Kahler metric to different complex structure}Assumptions and setting as in Lemma \ref{From the complex setting to the real setting of J}. Let
 \[\widetilde{J}_{u}=\psi_{u,\star}^{-1}\circ J_{u}\circ \psi_{u,\star};\ \widetilde{\omega}_{u}=\psi_{u}^{\star}\omega_{u} \in C_{1}(M_0,J_0).\] 
 For any $\delta$, there is an $\epsilon$ such that if $|u|<\epsilon$, then 
 there exists a $2$-form $\omega_{\phi,u} \in C_{1}(M_0,J_0)$ with the following properties.
\begin{itemize}
\item $\omega_{\phi,u}(\cdot,\widetilde{J}_{u} \cdot)$ is a K\"ahler metric with respect to $\widetilde{J}_{u} $;
\item  $ |\omega_{\phi,u}-\omega_{KE}|_{C^{k-1},M_0,g_0}\leq C\delta$, $\omega_{KE}$ is the K\"ahler-Einstein metric of the central fiber (gauge fixed).
\end{itemize}

\end{lem}
\begin{proof}{of Lemma \ref{perturbation of Kahler metric to different complex structure}:}\ Denote $\omega_{KE}\circ \widetilde{J}_{u}$ as the tensor 
\begin{equation*}
(\omega_{KE}\circ \widetilde{J}_{u})(X,Y)=\omega_{KE}(X,\widetilde{J}_{u}(Y)).
\end{equation*}

Consider the antisymmetric part of $\omega_{KE}\circ \widetilde{J}_{u}$ as 
\[(AS \omega_{KE}\circ \widetilde{J}_{u})(X,Y)=-\omega_{KE}(X,\widetilde{J}_{u} Y)+\omega_{KE}(Y,\widetilde{J}_{u} X).\]

Since $\omega_{KE}\circ J_0$ is symmetric, we have 
\begin{equation} 
-\omega_{KE}(X,\widetilde{J}_{s} Y)+\omega_{KE}(Y,\widetilde{J}_{s} X)
  =-\omega_{KE}[X,(\widetilde{J}_{s}-J_0) Y]+\omega_{KE}(Y,(\widetilde{J}_{s}-J_0) X).
\end{equation}

Using Lemma \ref{From the complex setting to the real setting of J}, by letting $\epsilon$ to be sufficiently small  we obtain
\begin{equation*}
|\widetilde{J}_{u}-J_{0}|_{C^{k},g_0,M_0}\leq \delta,\ \textrm{for all}
\ |u|<\epsilon.
\end{equation*}
Hence
\begin{equation}  |AS \omega_{KE}\circ \widetilde{J}_{u}|_{C^{k},g_0,M_0} \leq C \delta.
\end{equation}
Then we complexify $\omega_{KE}\circ \widetilde{J}_{u}$ with respect to $\widetilde{J}_{u}$. Using the fact that $\omega_{KE}^{2,0}\oplus \omega_{KE}^{0,2}$ is precisely the antisymmetric part of $\omega_{KE}\circ \widetilde{J}_{u}$, we get 
\begin{equation}\label{2,0 part small}  |\omega_{KE}^{2,0}|_{C^{k},g_0,M_0} +|\omega_{KE}^{0,2}|_{C^{k},g_0,M_0}\leq C \delta.
\end{equation}

 Since $\omega_{KE}=\widetilde{\omega}_{u}+d\theta$, we have
 \begin{equation}\label{2,0 part small 1}
 Re \omega_{KE}^{2,0}=Re(\omega_{KE}-\widetilde{\omega}_{u})^{2,0}=Re(d\theta)^{2,0}.
 \end{equation}
 Then, 
 \begin{equation}\label{2,0 part small 2}
 \omega_{KE}^{0,2}=\bar{\partial}_{u}\theta^{0,1},\ \omega_{KE}^{2,0}=\partial_{u}\theta^{1,0}.
 \end{equation}
By  (\ref{2,0 part small}), (\ref{2,0 part small 1}), and (\ref{2,0 part small 2}), we get 
 \begin{equation}
 |\bar{\partial}_{u}\theta^{0,1}|_{C^{k},g_0,M_0}+|\partial_{u}\theta^{1,0}|_{C^{k},g_0,M_0}\leq C\delta.
 \end{equation}
 We want to find a $(0,1)-$ form $\phi$ (with respect to $\widetilde{J}_{u}$) such that 
\begin{equation}\label{Want to find a good phi} |\phi|_{C^{k},g_0,M_0}\leq C\delta,\ \textrm{and}\ \bar\partial_{u} \phi=\omega_{KE}^{0,2}.\end{equation}
 This is straight forward by considering 
 \begin{equation}
 \bar{\partial}_{u}^{\star}\theta^{0,1}=-\widetilde{\omega}_{u}^{\bar{i}j}\theta_{\bar{i},j}.
\end{equation}  
By solving the Dirichlet problem 
\begin{equation}
\Delta_{u} f= \bar{\partial}^{\star}_{u}\theta^{0,1},
\end{equation}
we get a $f$ such that 
\begin{equation}
\bar{\partial}^{\star}_{u}(\theta^{0,1}-\bar{\partial}_{u}f)=0.
\end{equation}
 Let $\phi=\theta^{0,1}-\bar{\partial}_{u}f$, then 
 \begin{equation}\label{equation of phi}
 \bar{\partial}_{u}\phi=\bar{\partial}_{u}\theta^{0,1}=\omega_{KE}^{0,2},\ \bar{\partial}^{\star}_{u}\phi=0.
 \end{equation}
 \begin{clm}$|\phi|_{C^{k},g_0,M_0}\leq C\delta. $
 \end{clm}
 The claim easily follows from the existence of $G$, which is actually a smooth family of metrics. For the sake of a self-contained note, we include the detail here. We first show $|\phi|_{0}\leq C\delta$. Were this not true, there exists  a sequence $(\phi_{i},\ u_i)$ such that 
 \begin{equation}\label{dbar dstar goes to 0 and phi is 1}
 u_i\rightarrow u_{\infty},\ \bar{\partial}_{u_i}\phi_i\rightarrow 0
 \ \textrm{in the sense of}\ (C^{k},\ {g_0}),\ \bar{\partial}^{\star}_{u_i}\phi_i=0,\ \textrm{but}\ |\phi_i|_{C^{0}}=1. 
 \end{equation}

Hence
\begin{equation}\label{Hodge laplacian of phi i go to 0}
\Delta_{H,u_{i}}\phi_i\rightarrow 0\ \textrm{in the sense of}\ (C^{k-1},\ {g_0}).
\end{equation}
The  Bochner formula for any metric and  $1-$form $\eta$ reads as :
\begin{equation}\label{Bochner formula for 1-forms}
\Delta_{g}\eta =  \Delta_{H, g}\eta-Ric_{g}\odot\eta,\ \odot \ \textrm{ is some algebraic tensor product}.
\end{equation}
By standard Schauder estimates (with respect to $g_0$), and the fact that $|Ric_{u_i}|_{C^{k},g_0,M_0}\leq C\delta $ (from  item 2 in Definition \ref{Definition of the degeneration}), we deduce from (\ref{dbar dstar goes to 0 and phi is 1}), (\ref{Hodge laplacian of phi i go to 0}), and (\ref{Bochner formula for 1-forms}) that 
\begin{equation}\label{limit is nontrivial} |\phi_i|_{C^{k},g_0,M_0}\leq C. 
\end{equation}
Then $\phi_i\rightarrow \phi_{\infty}$ in $C^{k-1}$-topology such that 
\begin{equation}\label{limit is a nontrivial harmonic form, contradiction}
|\phi_{\infty}|_{0}=1,\ |\phi|_{C^{k-1},g_0,M_0}\leq C,\ 
\Delta_{H,\omega_{u_{\infty}}}\phi_{\infty}= 0.
\end{equation}
This contradict the simply connectness of Fano-manifolds (there should be no nontrivial harmonic form). Therefore 
 \begin{equation}
 |\phi|_{0}\leq C\delta.
 \end{equation}

By the Schauder-estimate of $g_{0}$ again  and the fact $|Ric_{u}|_{C^{k},g_0,M_0}\leq C\delta $, we get the following estimate with a $C$ independent of $u$.
\begin{equation}\label{phi is small in Ck}
|\phi|_{C^{k},g_0,M_0}\leq C \delta.
\end{equation}

Then $\phi$ satisfies (\ref{Want to find a good phi}). Moreover,
we have $\partial \bar{\phi}=\omega^{2,0}_{KE}$ and $|\bar{\phi}|_{C^{k},g_0,M_0}\leq C \delta$.

Define $\omega_{\phi,u}\triangleq \omega_{KE}-d(\phi+\bar{\phi})$. Then $\omega_{\phi,u}$ is K\"ahler with respect to $\widetilde{J}_{u}$. Furthermore, (\ref{equation of phi}) and (\ref{phi is small in Ck}) imply 
\[|\omega_{\phi,u}-\omega_{KE}|_{C^{k-1},g_0,M_0}\leq C\delta.\]

The proof is complete.
\end{proof}
 \begin{proof}{of Theorem \ref{thm convergence over degenerations}:}\ Without  loss of generality we  assume $B$ is the unit ball. Fix a finite cover $(\mathfrak{U}_{j}, 1\leq j\leq C)$ of $M_0$ in $\mathfrak{M}$ with nontrivial intersection with $M_{0}$. It suffices to show there exists  a $\omega_{\phi,u}$ as in Theorem \ref{thm convergence over degenerations} for any $u\in B_{r}$, where $r$ is sufficient small such that $B(r)$ is contained in the $B$-slice for all $\mathfrak{U}_{j}$, and every $u\in B(r)$ satisfies
\begin{equation} |\psi_{u,\star}^{-1}\circ J_{u}\circ \psi_{u,\star}-J_0|_{C^{k},g_0,M_0}+|\psi_{u}^{\star}g_{u}-g_0|_{C^{k},g_0,M_0}\leq \delta,
\end{equation} 
  guaranteed by Lemma \ref{From the complex setting to the real setting of J}.
   Then $\widetilde{J}_{u}(=\psi_{u,\star}^{-1}\circ J_{u}\circ \psi_{u,\star})$  and $\widetilde{g}_{u}(=\psi_{u}^{\star}g_{u})$ satisfy the assumptions of Lemma \ref{perturbation of Kahler metric to different complex structure}. Notice that since $C_{1}(M_0, \widetilde{J}_{u})$ is an integer cohomology class, $\widetilde{J}_{u}$ converges to  $J_{0}$,
   and $\widetilde{g}_{u}$ converges to $g_{0}\in C_{1}(M_0, J_{0})$, we have
   \[C_{1}(M_0, \widetilde{J}_{u})=C_{1}(M_0, J_{0}).\]
     Hence,   Lemma \ref{perturbation of Kahler metric to different complex structure} produces a metric $\omega_{\phi,u}$ which satisfies the assumptions in Theorem 1.1 of \cite{SW}. 
  
Step 1. We briefly show the K\"ahler-Ricci flow produces a degeneration of $\widetilde{J}_{u}$  to a K\"ahler-Einstein complex structure.  Actually, Theorem 1.1 of \cite{SW} says the K\"ahler-Ricci flow initiated from $\omega_{\phi,u}$ converges to a K\"ahler-Einstein metric $(\widetilde{\omega}_{KE}, \widetilde{J}_{KE})$. In particular, there exists a smooth time-dependent family of smooth diffeomorphisms $\psi_{t}$ such that $\psi_{t}^{\star}\widetilde{J}_{u}$ converges to 
  $ \widetilde{J}_{KE}$. Thus $\widetilde{J}_{u}$ degenerates to $\widetilde{J}_{KE}$, in the sense of the setting of Theorem 1.5 in \cite{CS}. 

Step 2. In this step we  point out that  $\widetilde{J}_{u}$ degenerates to $J_0$, which is also K\"ahler-Einstein.  To be precise, suppose $\sigma_{a}M_{u}=M_{u(a)}$, then $(M_{u(a)}, J_{u(a)})$ is biholomorphic to $(M_{u}, J_{u})$ via $\sigma^{-1}_{a}$. We consider the following diffeomorphism 
\begin{equation}\label{PSI a}
\Psi_{a}=(\psi_{u}^{-1}\circ \sigma_{a}^{-1}\circ \psi_{u(a)})^{-1}:\ M_{0}\rightarrow M_{0}. 
\end{equation}   
Then it's straightforward to deduce from Lemma \ref{From the complex setting to the real setting of J} and (\ref{PSI a}) that 
\begin{eqnarray}& &\nonumber
\lim_{a\rightarrow \infty}|\Psi_{a,\star}\circ \widetilde{J}_{u} \circ \Psi^{-1}_{a,\star}-J_0|_{C^{k},g_0,M_0}\nonumber
\\&=&\lim_{a\rightarrow \infty}|\psi_{u(a),\star}^{-1}\circ J_{u(a)} \circ \psi_{u(a),\star}-J_0|_{C^{k},g_0,M_0}\nonumber
\\&=& 0,
\end{eqnarray}
which means $\widetilde{J}_{u}$ degenerates to $J_0$.
 Finally, by Theorem 1.5  in \cite{CS} on uniqueness of K\"ahler-Einstein degeneration and the uniqueness of K\"ahler-Einstein metric due to Bando and Mabuchi \cite{Bando},   we conclude    
 \begin{eqnarray*}& &
 \widetilde{\omega}_{KE}(\textrm{limit of the K\"ahler-Ricci flow})
 \\&=&
 \omega_{KE}(\textrm{the original K\"ahler-Einstein metric in the central fiber}),
 \end{eqnarray*}
 up to a biholomorphism. 
\end{proof}
 \section{Appendix: Mukai 3-fold and proof of Theorem \ref{Main Theorem}.}

It suffices to check the Mukai 3-folds form a special degeneration in the sense of Definition \ref{Definition of the degeneration}. This is already done by the work of Tian \cite{Tian97} and Donaldson \cite{Don08}. Actually the Mukai 3-folds form a  holomorphic degeneration, which is much more special than the differential family in Definition \ref{Definition of the degeneration}.  To give a self-contained proof, we very briefly introduce the necessary aspects of Mukai-Umemura 3-folds for our application, where we use the convention employed by Tian in \cite{Tian97}.  Let 
$G(4,7)$ be the Grassmannian manifold consisting of the 4-planes in $\mathbb{C}^7$. For any 3-plane $F$ in $\wedge^2 \mathbb{C}^7$, we consider the subvariety $M_{F}$ as 
\begin{equation}
M_{F}=\{p\in G(4,7)|\ F|_{p}=0\}.
\end{equation}
Consider the 3-plane $F_0=\{f_1,f_2,f_3\}$, where
\begin{eqnarray*}
& &f_1=3e_1\wedge e_6-5e_2\wedge e_5+6e_3\wedge e_4;
\\& &f_2=3e_1\wedge e_7-2e_2\wedge e_6+e_3\wedge e_5;
\\& &f_3=e_2\wedge e_7-e_3\wedge e_6+e_4\wedge e_5.
\end{eqnarray*}

The Lie algebra of holomorphic vector fields over $M_{F_0}$, denoted as $\eta(M_{F_0})$, is isomorphic to $sl(2,\mathbb{C})$. As a subalgebra of $sl(7,\mathbb{C})$( $sl(7,\mathbb{C})$ naturally acts on 3-planes in $\wedge^2 \mathbb{C}^7$ via its fundamental action), $\eta(M_{F_0})$ is given by 
\begin{itemize}
\item \begin{displaymath}
H_1=\left |
\begin{array}{ccccccc}
3   &    0           &        0    &   0   &    0    &  0 & 0\\
      0      & 2    &   0         &   0   &   0     &  0 & 0\\
     0       &       0        &  1    &   0   &   0      & 0 & 0\\
      0      &    0           &    0        & 0     &   0     &  0 &0\\
      0      &  0    &    0    &   0            & -1      & 0 & 0\\
       0     &   0   &    0    &    0  &   0     &  -2 &0\\
        0      & 0     &    0    &  0    &  0      &   0 &-3 \\
\end{array}
\right |;
\end{displaymath}

\item  \begin{displaymath}
H_2=\left |
\begin{array}{ccccccc}
0   &    1           &        0    &   0   &    0    &  0 & 0\\
      0      & 0    &   1        &   0   &   0     &  0 & 0\\
     0       &       0        &  0    &   1  &   0      & 0 & 0\\
      0      &    0           &    0        & 0     &   1     &  0 &0\\
      0      &  0    &    0    &   0            & 0     & 1 & 0\\
       0     &   0   &    0    &    0  &   0     &  0 &1\\
        0      & 0     &    0    &  0    &  0      &   0 &0 \\
\end{array}
\right |;
\end{displaymath}
\item  \begin{displaymath}
H_3=\left |
\begin{array}{ccccccc}
0   &    0          &        0    &   0   &    0    &  0 & 0\\
      3      & 0    &   0        &   0   &   0     &  0 & 0\\
     0       &       5       &  0    &   0  &   0      & 0 & 0\\
      0      &    0           &    6        & 0     &   0     &  0 &0\\
      0      &  0    &    0    &   6            & 0     & 0 & 0\\
       0     &   0   &    0    &    0  &   5    &  0 &0\\
        0      & 0     &    0    &  0    &  0      &   3 &0 \\
\end{array}
\right |.
\end{displaymath}
\end{itemize}
Now we consider the deformation $F_{v}$ of $F_0$ as $F_{v}=\{f_1^{v},f^{v}_2,f^{v}_3\}$, such that 
\[f_i^{v}=f_i+\Sigma_{j+k\geq 7+i}v_{ijk}e_j\wedge e_k,\ i=1,2,3.\]
In \cite{Tian97}, Tian showed that when $v\neq 0$ and $|v|$ is small , the deformed Fano 3-fold $M_{F_{v}}$ does not admit any K\"ahler-Einstein metric. On the other hand, Donaldson showed  $M_{F_0}$ admits a K\"ahler-Einstein metric.  Moreover, Donaldson showed there is at least a 2-parameter family of small deformations $M_{\alpha,C}$ of $M_{F_0}$ such that $M_{\alpha,C}$ also admit K\"ahler-Einstein metric for all $\alpha,\ C$ small. Thus  $M_{F_0}$ is not isolated K\"ahler-Einstein. Nevertheless, $M_{F_0}$ is a K\"ahler-Einstein degeneration of $M_{F_{v}}$.  To be precise, we consider  
     the one real-parameter group generated by $H_1$. 
    \begin{displaymath}
\psi(s)=\left |
\begin{array}{ccccccc}
s^{-3}   &    0           &        0    &   0   &    0    &  0 & 0\\
      0      &s^{-2}     &   0         &   0   &   0     &  0 & 0\\
     0       &       0        &  s^{-1}     &   0   &   0      & 0 & 0\\
      0      &    0           &    0        & 1    &   0     &  0 &0\\
      0      &  0    &    0    &   0            & s     & 0 & 0\\
       0     &   0   &    0    &    0  &   0     &  s^{2} &0\\
        0      & 0     &    0    &  0    &  0      &   0 &s^{3}  \\
\end{array}
\right |.
\end{displaymath}
    It's easy to check that 
    \begin{equation}
    |s\psi(s)f_1^{v}-f_1|+|\psi(s)f_2^{v}-f_2|+|s^{-1}\psi(s)f_3^{v}-f_3|
    \leq C|s|.
    \end{equation}
    Hence \begin{equation}\label{degeneration of Tian's manifold to MU-3 fold }\lim_{s\rightarrow 0}\psi(s)M_{F_{v}}=M_{F_0}. 
    \end{equation}
    Thus $\Cup_{s\in [-1,1]} \psi(s)M_{F_{v}}$ is a smooth submanifold of $G(4,7)\times \mathbb{R}$.
    Moreover, since $C_{1}(M_{F_{v}})=C_{1}(Q)|_{M_{F_{v}}}$ ($Q$ is the universal quotient bundle of $G(4,7)$), we can take a closed positive $(1,1)$-form $\Omega\in C_{1}(Q)$ over $G(4,7)$. Then the restricted form $\Omega|_{\mathfrak{M}}$ plays the role of the $G$ in item 2 of Definition \ref{Definition of the degeneration}.

     Yuanqi Wang, Department of Mathematics, University of California at
Santa Barbara, Santa Barbara, CA, USA; wangyuanqi@math.ucsb.edu.

\end{document}